\documentclass{amsart}

\usepackage{datetime}
\usepackage{amsfonts}
\usepackage{amsmath}
\usepackage{amssymb}
\usepackage{amsthm}
\usepackage{eucal}
\usepackage{graphicx,color}
\usepackage{txfonts}
\usepackage[usenames,dvipsnames,svgnames,table]{xcolor}

\usepackage{hyperref}
\hypersetup{colorlinks,
            filecolor=black,
            linkcolor=MidnightBlue,
            citecolor=NavyBlue,
            urlcolor=RoyalBlue,
            bookmarksopen=true}

\theoremstyle{plain}
\newtheorem{lemma}{Lemma}

\newtheorem{theorem}{Theorem}
\newtheorem*{main}{Main Theorem}

\theoremstyle{remark}

\newtheorem{remark}{Remark}

\newcommand{\dn}{\check}
\newcommand{\Df}{\mathrm{DIFF}}

\newcommand{\Dt}{\frac{\partial}{\partial t}}

\newcommand{\ls}{\lesssim}
\newcommand{\mb}{\mathbb}
\newcommand{\mc}{\mathcal}

\newcommand{\mr}{\mathrm}
\newcommand{\rg}{\sigma}

\newcommand{\ve}{\varepsilon}
\newcommand{\vp}{\varphi}

\DeclareMathOperator{\Rm}{Rm}
\DeclareMathOperator{\Rc}{Rc}

\newcommand{\R}{\mathbb{R}}
\newcommand{\bS}{\mathbb{S}}

\newcommand{\Z}{\mathbb{Z}}

\newcommand{\sbs}{\subset}
\newcommand{\ra}{\rightarrow}

\newcommand{\met}{{\mr{MET}}}

\newcommand{\mq}{{\mr{MET}}^{1/4<\mr{sec}\leq1}}

\newcommand{\cl}[1]{{\mathcal{#1}}}
\newcommand{\0}[1]{_{_{#1}}}

\newcounter{mnotecount}[section]
\let\oldmarginpar\marginpar
\renewcommand\marginpar[1]{\-\oldmarginpar[\raggedleft\footnotesize #1]
{\raggedright\footnotesize #1}}

\begin{document}

\title{Sphere Bundles with $1/4$-Pinched Fiberwise Metrics}

\author{Thomas Farrell}
\address[Thomas Farrell]{Tsinghua University}
\email{farrell@math.tsinghua.edu.cn}

\author{Zhou Gang}
\address[Zhou Gang]{California Institute of Technology}
\email{gzhou@caltech.edu}

\author{Dan Knopf}
\address[Dan Knopf]{University of Texas at Austin}
\email{danknopf@math.utexas.edu}
\urladdr{www.ma.utexas.edu/users/danknopf}

\author{Pedro Ontaneda}
\address[Pedro Ontaneda]{Binghamton University}
\email{pedro@math.binghamton.edu}
\urladdr{www.math.binghamton.edu/pedro}

\thanks{TF and PO thank NSF for support in DMS-1206622.
ZG thanks NSF for support in DMS-1308985 and DMS-1443225.
DK thanks NSF for support in DMS-1205270.}

\begin{abstract}
We prove that all smooth sphere bundles that admit fiberwise $1/4$-pinched
metrics are induced bundles of vector bundles,
so their structure groups reduce from $\Df(\bS^n)$ to $\mr O(n+1)$.
This result implies the existence of many smooth
$\bS^n$-bundles over $\bS^k$ that do not support strictly $1/4$-pinched positively
curved Riemannian metrics on their fibers.
\end{abstract}

\maketitle
\setcounter{tocdepth}{1}
\tableofcontents

\section*{Introduction and statement of results}
Let $M$ be a smooth manifold. By a {\it smooth $M$-bundle
over a space $X$}, we mean a locally trivial bundle over $X$
whose structural group is $\Df(M)$. Here $\Df(M)$ is
the group of self-diffeomorphisms of $M$, 
with the smooth topology.

In this paper, we consider smooth sphere
bundles that can be equipped with a fiberwise
$1/4$-pinched metric, that is, a smoothly
varying Riemannian metric
on each fiber with sectional curvatures in the interval
$(1/4,1]$. The purpose of this work is to call attention to
the fact that such bundles are precisely $\mr O(n)$-bundles;
hence they are the induced bundles of vector bundles.

\begin{main}
For $n\geq3$, let $E\ra X$ be a smooth $n$-sphere
bundle over the locally compact space $X$. 
If $E$ can be equipped with a $1/4$-pinched fiberwise metric,
then $E$ equivalent to an $\mr O(n+1)$ $n$-sphere bundle over $X$.
\end{main}

Hence if the fibers of $E$ admit $1/4$-pinched metrics, then
the structural group of $E$ can be reduced from $\Df(\bS^n)$
to $\mr O(n+1)$.

\begin{remark}
Due to work of others, the assumption $n\geq3$ can be dropped from the
Main Theorem, because $\mr O(n+1)$ is a deformation
retract of $\Df(\bS^n)$ when $n=0,1,2,3$. This is due to Smale \cite{Smale}
for $n=2$, and to Hatcher \cite{Hatcher} for $n=3$.
\end{remark}

The proof of the theorem consists of two steps.
First we evolve the metrics on the fibers by normalized Ricci flow (\textsc{nrf})
to obtain fiberwise round metrics. Note that we cannot use standard
results here, because we need uniform control on how far
each metric moves in the orbit of the diffeomorphism group.
Instead, we adapt methods developed in \cite{GKS, GK}
for controlling coordinate parameterizations of mean curvature
flow solutions. Then we use a fiberwise
version of the fact that each round metric is isometric
to the canonical round metric on the sphere.
We explain these steps below in more detail, starting
with the second.

\begin{remark}
Suppose $n$ is an odd integer, $k$ is divisible by $4$, and $n\gg k$ (it suffices that $n\geq 3k+4$).
Then it is known \cite{FH78} that $\pi_{k-1}(\Df(\bS^n))=\Z\oplus\Z\oplus\{\mbox{ finite }\}$, and 
$\pi_{k-1}\mr O(n+1)=\Z$. Consequently, the Main Theorem
implies the existence of many smooth $\bS^n$-bundles over $\bS^k$ that do not support
strictly $1/4$-pinched positively curved Riemannian metrics on their fibers. Further examples
result from Theorem~3$^{\prime\prime\prime}$ (p.~59) of \cite{FO}. Namely, for every pair
of positive integers $k$ and $n$, with $n\equiv (3-k)\mod 4$
and $n\geq n_1(k)$, where $n_1(k)$ is an integer depending on $k$, there exists such a smooth bundle.
\end{remark}

\part{Topology}

\section{Notations and definitions}

Let $E_1\ra X$ and $E_2\ra X$ be two smooth $M$-bundles
over the space $X$. Let $f:E_1\ra E_2$ be a {\it fiberwise map}, 
meaning that $f$ sends the fiber in $E_1$ over each $x\in X$ diffeomorphically
onto the fiber in $E_2$ over $x$. The expression of $f$ in trivializing charts over an open
set $U\sbs X$ has the form $(x,p)\mapsto (x,f(x,p))$ for all $(x,p)\in U\times M$.
This determines a map $U\ra \Df(M)$.

We say that the bundles $E_1$, $E_2$ are {\it smoothly
equivalent} (or simply {\it equivalent}) if there is a fiberwise
map $f:E_1\ra E_2$ such that all induced
maps $U\ra \Df(M)$ are continuous. The map $f$ is a
{\it (smooth) bundle equivalence}.

The space of Riemannian metrics on the smooth manifold
$M$, with the smooth topology, will be denoted by $\met(M)$.
Let $E\ra X$ be a smooth $M$-bundle
over the space $X$. We write $\pi:E\ra X$ for the bundle
projection. A {\it (smooth) fiberwise metric on $E$} is a collection
$\{g_x\}_{x\in X}$ of Riemannian metrics on the fibers of $E$ that
varies smoothly on $X$. By this we mean that $g_x$ is a Riemannian
metric on the fiber $E_x=\pi^{-1}(\{x\})$, and
if $\alpha:\pi^{-1}(U)\ra U\times M$ is a trivializing
chart, then the maps $U\ra \met(M)$, $x\mapsto \alpha_*g_x$ are
continuous.

We denote by $\mq(\bS^n)$ the space of Riemannian metrics on
the $n$-sphere $\bS^n$ with sectional curvatures in the interval
$(1/4,1]$, considered with the smooth topology.
Similarly, we denote by $\met^{\mr{round}}(\bS^n)$ the space of 
metrics on $\bS^n$ with all sectional curvatures equal to $1$,
i.e., the space of round metrics, again with the smooth topology.
The canonical round metric on $\bS^n$, that is, the one induced by the inclusion
$\bS^n\sbs\R^{n+1}$, is denoted by $\sigma_n$, or simply $\sigma$.

\section{Ricci flow and families of round metrics}

The following is a ``family version'' of the fact that for every
$g\in\met^{\mr{round}}(\bS^n)$, there is an isometry
$\phi_g:(\bS^n, g)\ra(\bS^n,\sigma)$.

\begin{lemma}	\label{FO1}
There is a continuous map
\[
\Phi:\met^{\mr{round}}(\bS^n)\longrightarrow \Df(\bS^n)
\]
such that $\Phi(g):(\bS^n, g)\ra (\bS^n,\sigma)$ is an isometry.
\end{lemma}

\begin{proof}
Fix $x\in\bS^n$ and let $\cl{B}_\sigma$ be a $\sigma$-orthonormal
basis for $T_x\bS^n$. For any $g\in\met^{\mr{round}}(\bS^n)$,
the Gram--Schmidt process produces
a $g$-orthonormal basis $\cl{B}_g$ of $T_x\bS^n$. Moreover
the map $g\mapsto \cl{B}_g$ is continuous. Now using the
frame $\cl{B}_g$ and the
exponential maps $\exp^\sigma$ and $\exp^g$ in the usual way,
we get an isometry $\phi:(\bS^n, g)\ra (\bS^n,\sigma)$.
We put $\phi=\Phi(g)$. It is straightforward to check that
the map $\Phi$ is continuous.
\end{proof}

Henceforth we write $\phi_g=\Phi(g)$.
\smallskip

Let $g_0$ be a $1/4$-pinched metric on $\bS^n$. Brendle and
Schoen \cite{BS} prove that if $g(t)$ is normalized Ricci flow
(\textsc{nrf}) with initial data $g(0)=g_0$, then $g(t)$ converges in the
smooth topology as $t\rightarrow\infty$ to a round metric $g^\bullet$.

The key fact that we need to prove the Main Theorem is as follows.

\begin{theorem}	\label{Continuous}
Let $X$ be compact, and let $F:X\ra \mq(\bS^n)$ continuous.
Then the map $F^\bullet:X\ra \met^{\mr{sec=1}}(\bS^n)$ induced by \textsc{nrf}
is also continuous, where $F^\bullet(g)=(F(g))^\bullet$.
\end{theorem}

We prove Theorem~\ref{Continuous} in Part~\ref{Analysis} below.

\section{Proof of the Main Theorem}

Using Theorem~\ref{Continuous}, we prove our Main Theorem as follows.

\begin{proof}[Proof of Main Theorem]
Let $E\ra X$ be a smooth $\bS^n$-bundle. We write $\pi:E\ra X$ for
the bundle projection, and set $E_x=\pi^{-1}(x)$.
We assume that there is a fiberwise $1/4$-pinched metric $\{g_x\}_{x\in X}$ on $E$.

Let $\{(U_i, \alpha_i)\}_{i\in\mc I}$ be an atlas of trivializing bundle charts of $E$.
Since $X$ is locally compact, we can assume that all $\bar U_i$ are compact.
For the change of charts, we write
$\alpha_{ij}:U_{ij}\times\bS^n\ra U_{ij}\times\bS^n$, $\alpha_{ij}=\alpha_j\circ\alpha_i^{-1}$. Here
$U_{ij}=U_i\cap U_j$. We write $\alpha_{ij}(x,p)=(x,\alpha_{ij}(x)(p))$.
Note that 
\begin{equation}	\label{1}
\alpha_{ij}(x):\big(\,\bS^n\,,\,(\alpha_i)_*g_x\,\big)\,\,\longrightarrow\,\,\big(\,\bS^n\,,\,(\alpha_j)_*g_x\,\big)
\end{equation}
is an isometry.

Now we evolve each metric $g_x$ on $E_x$ by Ricci flow to obtain a
fiberwise metric $\{g_x^\bullet\}_{x\in X}$ on $E$ such that each
$g_x^\bullet$ is a round metric. Note that the family $\{g_x^\bullet\}_{x\in X}$
is smooth: on each $U_i$, Theorem~\ref{Continuous} ensures that
the map  $U_i\ra \met^{\mr{sec}=1}(\bS^n)$ given by
$x\mapsto (\alpha_i)_*g_x^\bullet$ is continuous. 
Furthermore, since Ricci flow preserves isometries,
we find that all the maps $U_i\ra \met^{\mr{sec}=1}(\bS^n)$ patch 
to give the round fiberwise metric $\{g_x^\bullet\}_{x\in X}$
on $E$. Define $G_i:U_i\ra\met^{sec=1}(\bS^n)$ by
$G_i(x)=(\alpha_i)_*g_x^\bullet$.
As in (1), we similarly obtain that
\begin{equation}	\label{2}
\alpha_{ij}(x):\big(\,\bS^n\,,\,(\alpha_i)_*g^\bullet_x\,\big)\,\,\longrightarrow\,\,\big(\,\bS^n\,,\,(\alpha_j)_*g^\bullet_x\,\big)
\end{equation}
\noindent is an isometry.

Next we apply the map $\Phi$  from Lemma 1 to each $G_i$ to obtain
$\Phi\circ G_i:U_i\ra \Df(\bS^n)$, with
\[
x\mapsto \phi\0{(\alpha_i)_*g_x^\bullet}.
\]
Let $f_i:U_i\times\bS^n\ra U_i\times\bS^n$
be given by 
\begin{equation}	\label{3}
f_i(x,v)=\big(\,x\,,\,\phi\0{(\alpha_i)_*g_x^\bullet}(v)\,\big).
\end{equation}
Finally for each $i, j$, define $\beta_{i,j}=f_j\circ \alpha_{ij}\circ
f_i^{-1}$. Then, by definition, the following diagram
commutes:
\begin{equation}	\label{4}
\begin{array}{ccc}
U_{ij}\times \bS^n&\xrightarrow{\quad f_i\quad}&
U_{ij}\times \bS^n\\{\mbox{\scriptsize $\alpha_{ij}$}}\downarrow
&&\downarrow {\mbox{\scriptsize $\beta_{ij}$}}\\
U_{ij}\times \bS^n&\xrightarrow{\quad f_j\quad}&
U_{ij}\times \bS^n\end{array}
\end{equation}
Therefore the $f_i$ patch to give a smooth bundle equivalence
$f:E\ra E'$, where $E'$ is the smooth $\bS^n$-bundle
constructed using the $\{\beta_{ij}\}_{i,j}$. 
Write $\beta_{ij}(x,p)=(x,\beta_{ij}(x)(p))$. 
Then, by definition of the $\beta_{ij}$ (see diagram (4)) and the definition 
of the $f_i$ given in (3), we have
\begin{equation} 	\label{5}
\beta_{ij}(x)\quad=\quad\phi\0{(\alpha_j)_*g_x^\bullet}\quad\circ\quad\alpha_{ij}
\quad\circ\quad\phi\0{(\alpha_i)_*g_x^\bullet}^{-1}.
\end{equation}

But by Lemma~\ref{FO1}, both maps $\phi$ in~\eqref{5} are isometries,
and the map $\alpha_{ij}$ appearing in the middle of~\eqref{5} is the same 
map from equation~\eqref{2}, so is also an isometry.
Hence
$$\beta_{ij}(x):(\bS^n,\sigma)\longrightarrow(\bS^n,\sigma)$$
\noindent is an isometry. So $\beta_{ij}(x)\in\mr O(n+1)$.
This proves that the bundle $E'$ is an $\mr O(n+1)$ bundle.
\end{proof}

\part{Analysis}	\label{Analysis}

\section{Outline of the argument} 

Let $\mc S_2$ denote the space of smooth sections of the bundle of symmetric
$(2,0)$-tensor fields over $\bS^n$, and let $\mc S_2^+$ denote the open convex
cone of positive-definite tensor fields (i.e.~Riemannian metrics). Then $\mc S_2$ with the
$C^\infty$ topology is a Frech\'et space. There is a natural right action of the group
$\mc D:=\Df(\bS^n)$
of smooth diffeomorphisms of $\bS^n$ on $\mc S_2^+$ given by $(g,\psi)\mapsto\psi^*g$.
So it is often convenient to regard $\mc S_2^+$ as a union of infinite-dimensional orbits
$\mc O_g$.

Ebin's slice theorem \cite{E68} reveals that $\mc S_2^+$ almost has the structure of an
infinite-dimensional manifold with an exponential map. More precisely, the theorem states that for any
metric $g$, \textsc{(i)} there exists a map $\chi:\mc U\rightarrow\mc D$ of a neighborhood $\mc U$ of
$g$ in $\mc O_g$ such that $\big(\chi(\psi^*g)\big)^*g=\psi^*g$ for all $\psi^*g\in\mc U$; and
\textsc{(ii)} there exists a submanifold $\mc H$ of $\mc S_2^+$ containing $g$ such that the map
$\mc U\times\mc H\rightarrow\mc S_2^+$ given by
\[
(\psi^*g,h)\mapsto\big(\chi(\psi^*g)\big)^*h
\]
is a diffeomorphism onto a neighborhood of $g$ in $\mc S_2$. Infinitesimally, there
is an orthogonal decomposition $T_g\mc S_2=\mc V_g\oplus\mc H_g$, where
$\mc V_g$ is the image of the Lie derivative map (i.e., infinitesimal diffeomorphisms)
and $\mc H_g$ is the kernel of the divergence map.

\subsection{The first step}
Let $\Xi$ be a compact set of strictly $\frac14$-pinched metrics on $\bS^n$,
 $(n\geq3)$, and let $\omega_n=2\pi^{(n+1)/2}/\Gamma((n+1)/2)$ denote the volume of the
canonical round metric $\rg$ induced by the embedding $\bS^n\subset\mb R^{n+1}$.
Noting that the map
\[
\nu:g\mapsto\frac{\omega_n}{\int_{\bS^n}\mr d\mu[g]}\,g
\]
depends continuously on $g$, we consider the compact set $\Xi_*:=\nu(\Xi)$ of
volume-normalized metrics. We evolve each metric $g_0\in\Xi_*$ by normalized
Ricci flow (\textsc{nrf}),
\begin{subequations}
\begin{align}	\label{NRF}
\Dt g&=-2\Rc+\frac{2}{n}\bar R\,g,\\
g(0)&=g_0,
\end{align}
\end{subequations}
where 
\[
\bar R(t):=\fint_{\bS^n}R\,\mathrm d\mu
:=\frac{\int_{\bS^n} R\,\mathrm{d}\mu}
        {\int_{\bS^n}\,\mathrm{d}\mu}
\]
denotes the average scalar curvature of the evolving metric.

Let $\mc O_{\rg}:=\mc D(\rg)$ denote the orbit of the standard metric.
Work of Brendle and Schoen \cite{BS} implies that every solution of \textsc{nrf}
originating in $\Xi_*$ converges in the $C^\infty$ topology to an element of $\mc O_{\rg}$.
Given constants $\alpha_0,\dots,\alpha_\ell>0$, we define a neighborhood of $\mc O_{\rg}$ in
the space of metrics on $\bS^n$ by
\[
\Sigma_{\alpha_0,\dots,\alpha_\ell}
:=\Big\{\psi^*g:\psi\in\mc D\mbox{ and }
\|\nabla^j\big(\Rm[g]-\Rm[\rg]\big)\|_g\leq\beta_j
\mbox{ for all } 0\leq j\leq \ell\Big\}.
\]
The norms above are in $L^\infty$ measured with respect to $g$. It follows
from \cite{BS}, using standard facts about continuous dependence of solutions of parabolic
equations on their initial data, that there exists a time $T_*$ depending only on $\Xi_*$
and $\alpha_0,\dots,\alpha_\ell$ such that every solution $g(t)$ of \textsc{nrf} originating in
$\Xi_*$ belongs to $\Sigma_{\alpha_0,\dots,\alpha_\ell}$ for all times $t\geq T_*$.

We now derive a uniform bound on how far solutions originating in $\Xi_*$ can move in the
direction of the $\mc D$-action up to the finite time $T_*$. Indeed,  because \cite{BS}
implies that solutions $g(t)$ originating in $\Xi_*$ exist for all positive time, it follows
from standard short-time existence results for Ricci flow, again using continuous dependence
on initial data, that there exists $C_*$ such that $\|\Rm(\cdot,t)\|_{g(t)}\leq C_*$ for all solutions
$g(t)$ originating in $\Xi_*$ and all times $0\leq t\leq T_*$. This gives a uniform bound for
the \textsc{rhs} of \eqref{NRF}. It follows, again using the slice theorem, that all metrics $g(T_*)$
belong to the set
\[
\Sigma_{\alpha_0,\dots,\alpha_\ell,}^C
:=\Big\{\psi^*g:\psi\in\mc D,\,
\|\nabla^j\big(\Rm[g]-\Rm[\rg]\big)\|_g\leq\beta_j,\,
\|\psi-\mr{id}\|_{L^\infty}\leq C\Big\},
\]
and that $g(T_*)$ depends continuously on $g(0)=g_0\in\Xi_*$.

\subsection{The second step}
To control how far solutions originating in $\Xi_*$ can move in the direction of the
$\mc D$-action for all time, we craft a modified linearization argument. Note that
while  $\Sigma_{\alpha_0,\dots,\alpha_k}$ is a neighborhood of the set $\mc O_{\rg}$
of stationary solutions of \textsc{nrf}, that set is infinite-dimensional, and standard stability
theorems do not function in the presence of an infinite-dimensional center manifold.

The fact that Ricci flow is not strictly parabolic, again due to its invariance under
the $\mc D$-action, in another obstacle to using standard linearization techniques.
One obtains parabolicity by fixing a gauge, using the well-known DeTurck trick.
While this remedies one problem, it creates another: the gauge-fixed flow introduces
instabilities in the (formerly invisible) direction of the $\mc D$-action. We describe below
how we remedy this.

If $\hat g$ is a chosen element of $\mc O_{\rg}$ with Levi-Civita connection $\hat\Gamma$,
one defines a $1$-parameter family of vector fields $\hat W$ on $(\bS^n,g)$ by
\[
\hat W^k=g^{ij}\hat A_{ij}^k,
\]
where $\hat A$ is the tensor field $\Gamma-\hat\Gamma$.
The decorations indicate that $\hat W$ and $\hat A$ depend on our choice of
gauge $\hat g\in\mc O_{\rg}$  The normalized Ricci--DeTurck flow (\textsc{nrdf})
is the strictly parabolic system
\begin{equation}	\label{NRDF}
\Dt g=-2\Rc+\mc L_{\hat W}(g)+\frac{2}{n}\bar R g,
\end{equation}
where $\mc L$ denotes the Lie derivative. 

We henceforth consider \textsc{nrdf} with initial data
$g(T_*)\in \Sigma_{\alpha,\,\beta_0,\dots,\beta_k}^C$.
By a slight abuse of notation, we continue to denote the evolving
metrics by $g(t)$.

Standard variational formulas imply that the linearization of \eqref{NRDF} at $\hat g$ is the
autonomous, self-adjoint, strictly parabolic system
\begin{equation}	\label{GeneralLinear}
\Dt h = \Delta_\ell h + 2(n-1)\left\{h-\frac{1}{n}\bar H g\right\},
\end{equation}
where $H=g^{ij}h_{ij}$ is the trace of the perturbation $h$, and
$\bar H =\fint_{\bS^n}H\,\mathrm d\mu$ is its average over the manifold. The operator $\Delta_\ell$
denotes the Lichnerowicz Laplacian acting on symmetric $(2,0)$-tensor fields. In coordinates,
\[
\Delta_{\ell}h_{ij}=\Delta_2 h_{ij}+2\hat R_{ipqj}h^{pq}-\hat R_i^k h_{kj}-\hat R_j^k h_{ik},
\]
where $\Delta_2$ is the rough Laplacian acting on $(2,0)$-tensors,
$\Delta_2 h_{ij}=g^{pq}\nabla_p\nabla_q h_{ij}$.

Let $h^\circ$ denote the trace-free part of $h$, defined by
\[
h^\circ=h-\frac{1}{n}Hg.
\]
Then, using the structure of the curvature operator $\hat{\Rm}$, one can rewrite \eqref{GeneralLinear} as
\begin{equation}	\label{Linear}
\Dt h = Lh := \Delta_2 h -2h^\circ+2\frac{n-1}{n}(H-\bar H)g.
\end{equation}
To proceed, we decompose the operator $L$ defined in~\eqref{Linear} as
\[
L h = \frac{1}{n}L_0 H + L_2 h^\circ,
\]
where
\begin{align*}
L_0 H &=\Delta_0 H + 2(n-1)(H-\bar H),\\
L_2 h^\circ&=\Delta_2 h^\circ-2h^\circ.
\end{align*}
Here $\Delta_0$ is the Laplace--Beltrami operator acting on scalar functions.

It is evident that $L_2$ is strictly stable. Using the $L^2$ inner product shows that
\[
(L_0 H,H)=-\|\nabla H\|^2 + 2(n-1)\big(\|H\|^2-\omega_n\bar H^2\big).
\]
So the spectrum of $L_0$ is bounded from above by that of $\Delta_0+2(n-1)$.
Recall that the spectrum of $-\Delta_0$ is $\{j(n+j-1)\}_{j\geq0}$.
The eigenspace $\Phi_j$ consists of the restriction to
$\bS^n\subset\mathbb R^{n+1}$ of $j$-homogeneous polynomials
in the coordinate functions $(x^1,\dots,x^{n+1})$. Hence $\Phi_0$
consists of constant functions. And if $H$ belongs to $\Phi_1$,
then one has $\bar H=0$ and hence
\[
(L_0 H,H)=\{-n+2(n-1)\}\|H\|^2=(n-2)\|H\|^2.
\]
Because $\|h\|^2=\frac{1}{n}\|H\|^2+\|h^\circ\|^2$, we have proved the following.
(Compare \cite{K09}.)

\begin{lemma}\label{LM:spectral}
For all $n\geq3$, the operator $-L$ defined in \eqref{Linear} is unstable.

It has a single unstable eigenvalue $2-n$ with $(n+1)$-dimensional eigenspace consisting
of infinitesimal conformal diffeomorphisms $\vp\hat g$.

Its $1$-dimensional null eigenspace is $\{c\hat g:c\in\mathbb R\}$.

Its remaining eigenspaces are strictly stable, with eigenvalues $\geq4$.
\end{lemma}

For times $t\geq T_0=T_*$, we decompose solutions $g(t)$ of \textsc{nrdf} as
\begin{align}
g(t)=\left\{a(t)+\sum_{i=1}^{n+1}b_i(t)\vp_i\right\}\hat g + \check g,\label{eq:decomG}
\end{align}
where $\vp_i$ are the infinitesimal conformal diffeomorphisms, and $\check g$ belongs
to the strictly stable spectrum of $-L$. Note that $a(t)$ is a redundant parameter: the
fact that $g(t)$ has volume $\omega_n$ means it can be computed from the
other elements of the decomposition. Indeed,  as we now show, the parameter $a$ is
determined by $\check{g}$ and $b_k$ $(k=1,2,\dots,n+1)$.

\begin{lemma}\label{LM:redA}
If $\|\check{g}\|$, $|a-1|$, and $\sum_{k}|b_k|$ are sufficiently small, then $a$ is uniquely determined,
and\footnote{We write $|F|\ls|G|$ if there exists a uniform constant $C$ such that
$|F|\leq C|G|$.}
\begin{equation}
|a-1|\ls \sum_{k}|b_k|^2+\|\check{g}\|_{\infty}^2.\label{eq:meas}
\end{equation}
\end{lemma}

\begin{proof}
Recall that we normalized $\mr{Vol}(g)=\mr{Vol}(\sigma)=\omega_n$,
and that in~\eqref{eq:decomG}, we decomposed $g$ as 
\begin{equation}	\label{eq:decommg}
g=\big(a+b_i x^i\big)\hat{g}+\dn{g},
\end{equation}
where $\check g$ is a chosen element of $\mc O_\sigma$, and
$\check g\perp \{\hat{g},\ x^1\hat{g},\dots,x^{n+1}\hat g\}$.
Using the standard fact that $\frac{d}{d\ve}\big\{\log\det(\hat g+\ve h)\big\}=\hat g^{ij}h_{ij}$, we rewrite
$ \sqrt{\det g}$, obtaining
\begin{equation}	\label{eq:expandDetG}
\sqrt{\det g}=a^{\frac{n}{2}}\sqrt{\det \hat g}\,\Big\{1+\frac12\hat g^{-1}h+\mc O(\|h\|_{\infty}^2)\Big\},
\end{equation}
where
\[
h:=\frac{1}{a}\big\{b_i(t)x^i\hat g+\check g\big\}.
\]
Consequently, writing $\mr d\mu\equiv\mr d\mu[\sigma]$, one has
\[
\int \sqrt{\det g}\,\mr d\mu=a^{\frac{n}{2}}
\Big\{\int \sqrt{\det\hat g}\,\mr d\mu+\frac12\int \hat g^{-1}h\,\mr d\mu
+\int \mc O(\|h\|_{\infty}^2)\,\mr d\mu\Big\}.
\]
Using the orthogonality condition $\check g\perp\hat g$ and volume normalization, this becomes
\[
\omega_n=a^{\frac{n}{2}}\omega_n\big\{1+\|h\|_\infty^2\big\},
\]
which implies estimate~\eqref{eq:meas}.

To see that $a$ is unique, we regard $a$ in \eqref{eq:decommg} as an independent parameter, and
prove that $\big|\partial_a \int \sqrt{\det g}\,\mr d\mu\big|\geq c$ for a constant $c>0$ provided that $|a-1|$ is small.
Specifically, computing as above, we find that
\[
\frac{\partial}{\partial a} \int \sqrt{\det g}\,\mr d\mu=\omega_n
+\mc O\left(\sum_i |b_i| + |a-1| + |(\hat g)^{-1}\check{g}|\right).
\]
So $\int \sqrt{\det g}\,\mr d\mu$ depends almost linearly on $a$, provided that 
$\sum_i|b_i|+|a-1|+|(\hat g)^{-1}\check{g}|$ is small. This and the normalization
$\int \sqrt{\det g}\,\mr d\mu=\omega_n$ imply that $a$ is unique.
\end{proof}

Our modified linearization argument follows ideas introduced in \cite{GKS} and \cite{GK}.
At a sequence of times $T_\gamma\nearrow\infty$, with $T_0=T_*$, we
construct  DeTurck background metrics $\hat g_\gamma$ such that all $b_i(T_\gamma)=0$, and
all $b_i(t)$ remain small for $T_\gamma\leq t\leq T_{\gamma+1}$. We obtain estimates
that prove that \textsc{(i)} the sequence $\hat g_\gamma$ converges, and \textsc{(ii)}
each solution $g(t)$ of \textsc{nrf-nrdf} converges to an element of
$\mc O_{\rg}$ that depends continuously on $g_0\in\Xi_*$. The details follow.

\section{Ricci--DeTurck flow for a sequence of background metrics}	\label{MainConstruction}
In the remainder of this paper, the norms $\|\cdot \|_{2}$, $\|\cdot\|_{\infty}$, the inner product
$\langle \cdot,\cdot\rangle$, and the definition of orthogonality $\perp$ should be understood with respect
to the metric $\sigma$ induced on the unit sphere in $\mb R^{n+1}$.
\smallskip

In this section, we study \textsc{nrdf} defined in \eqref{NRDF},
\[
\Dt g=-2\Rc+\mc L_{\hat W}(g)+\frac{2}{n}\bar R g.
\]
As is well known, this equation is related to normalized Ricci flow,
\[
\Dt \tilde{g}=-2\Rc[\tilde{g}]+\frac{2}{n}\bar R[\tilde g] \tilde{g}
\]
by the relation $\phi_t^{*}g(t)=\tilde{g}(t)$,
where $\phi_t$ is the one-parameter family of diffeomorphisms
$\phi_t\colon\mathbb{S}^n\rightarrow \mathbb{S}^n$
defined by 
\begin{align}
\Dt \phi_{t} (p)=&-\hat W(\phi_t(p),t),\\
\phi_0=&\mr{Id}.
\end{align}

It is well known that $\Rc[g]$ is diffeomorphism invariant.
Taking advantage of this fact, we will study solutions of the system
\begin{subequations}
\begin{align}	\label{NRDF2}
\Dt g=&-2\Rc+\mc L_{\hat W}(g)+\frac{2}{n}\bar R g,\\
g(T_0)=& \phi^* g_0,
\end{align}
\end{subequations}
noting  its evolution is identical to \eqref{NRDF}. Here $\phi$ is a diffeomorphism
to be chosen. That choice will be the central topic in what follows.

We choose $\phi$ from an admissible class of diffeomorphisms. 
Intuitively, a diffeomorphism is admissable if it is ``almost affine.'' More precisely,
for any vector $\vec\epsilon=(\epsilon_1,\dots,\epsilon_{n+1})\in \mb R^{n+1}$,
with $\|\vec\epsilon\|$ sufficiently small, we say a diffeomorphism
$\phi(\vec\epsilon):\mb S^n\rightarrow\mb S^n$ is \emph{admissible} if
\begin{equation}	\label{admissible}
\phi^*(\vec\epsilon)g=\Big(1+\sum_{k=1}^{n+1}\epsilon_k x^k\Big)g+\mc O\big(\|\epsilon\|^2\big).
\end{equation}
Existence of admissible diffeomorphisms follows easily from the implicit function theorem.

Now we are ready to state the main theorem of this section.
Recall that 
\begin{align}	\label{eq:g0Decom} 
g_0=\left\{a(0)+\sum_{k}b_k(0)x^{k}\right\}\hat{g}+\check {g}(0),
\end{align}
with $\check{g}(0)\perp \hat{g},\ x^{k}\hat{g}$ for $k=1,\dots,n+1$.
Define a constant $\delta_0$ by
\begin{equation}	\label{def:delta01}
\delta_0:=|1-a(0)|+\sum_{k=1}^{n+1}|b_k(0)|
+\sum_{|\alpha|=0}^{2n} \|\nabla^\alpha \check{g}(0)\|_2.
\end{equation}
 We shall prove:
 
\begin{theorem}\label{THM:conver}
There exist an element $\hat g\in\mc O_\sigma$ and
an admissible diffeomorphism $\phi_{\infty}=\phi(\vec\epsilon_{\infty})$,
both depending continuously on $g_0$,
with $\vec\epsilon_{\infty}\in \mathbb{R}^{n+1}$ and $\|\vec\epsilon_{\infty}\|\ll 1$,
such the the solution $g_\infty(t)$ of the system
\begin{subequations}
\begin{align}
\Dt g_\infty(t)=&-2\Rc[g_\infty(t)]+\mc L_{\hat W}[g_\infty(t)]+\frac{2}{n}\bar R[g_\infty(t)]\,g_\infty(t),\\
g_\infty(0)=& \phi^*_{\infty} g(T_0),
\end{align}
\end{subequations}
satisfies the estimate
\begin{equation}
\|g_\infty(t)-\hat{g}\|_{\infty}\leq C \delta_0 e^{-\frac{3}{2}t}
\end{equation}
for some $C<\infty$, and thus has the property that
\begin{equation}
\hat{g}=\lim_{t\rightarrow \infty}g_\infty(t)\quad\mbox{in }L^\infty.
\end{equation}

\end{theorem}
We will divide the proof into several steps that appear below.
We note that Theorem~\ref{Continuous} follows easily from Theorem~\ref{THM:conver}
and the estimates obtained in its proof.

The general strategy is to approximate $\vec{\epsilon}_{\infty}$ by finding an
``optimal'' admissible diffeomorphism, associated with $\vec\epsilon_{T_\gamma}$, for each time $T_\gamma$,
and proving that at least a subsequence of $\{\vec\epsilon_{T_\gamma}\}$ is Cauchy.
Then we define $\vec\epsilon_{\infty}$ as the limit of that subsequence.

For simplicity of notation and where no confusion will result, we henceforth ignore the index $\gamma$, 
making the identifications $0=T_\gamma$, $T=T_{\gamma+1}$, and writing $g(t,T)=g_T(t)=g_{T_\gamma}(t)$.
Thus we consider the system
\begin{subequations}
\begin{align}
\Dt g(t,T)=&-2\Rc[g(t,T)]+\mc L_{\hat{W}} [g(t,T)]+\frac{2}{n}\bar{R}[g(t,T)]\,g(t,T),	\label{eq:deTurck}\\
g(0,T)=&\phi^*_{T} g_0,	\label{eq:initial-cond}
\end{align}
\end{subequations}
in a time interval $t\in[0,T]$. In Section~\ref{sec:LMendP}, we prove the following result.

\begin{lemma}\label{LM:endPoint}
There exists a (possibly small) constant $M\in (0,\infty]$ such that for any
$T\in [0,M]$ (or $[0,\infty)$ if $M=\infty$), there exists a unique admissible diffeomorphism
$\phi_{T}=\phi(\vec\epsilon_{T})\colon\mb{S}^n\rightarrow \mb{S}^n$ such that for any time
$t\in [0,T]$, the solution of~\eqref{eq:deTurck} with initial condition~\eqref{eq:initial-cond} decomposes as
\begin{equation}	\label{eq:endPo}
g(t,T)=\Big\{a(t,T)+\sum_{k} b_k(t,T) x^k\Big\}\hat{g}+\check{g}(t,T),
\end{equation} with 
\[
b_k(T,T)=0\quad\mbox{for all}\quad k=1, \dots, n+1,
\]
and for all $t\in [0,\ T]$,
\begin{align}
\check{g}(t,T)\perp \hat{g},\,x^{k}\hat{g}.
\end{align}
\end{lemma}

Henceforth, we call the admissible diffeomorphism $\phi_{T}=\phi(\vec\epsilon_{T})$
constructed in Lemma~\ref{LM:endPoint} the \emph{optimal admissible diffeomorphism} at time $T$.
In what follows, we use bootstrap arguments to  derive estimates for the solution $g(t,T)$ for $t\in[0,M]$,
where $M$ is the constant introduced in Lemma~\ref{LM:endPoint}.
\smallskip

By the construction of the diffeomorphism $\phi_0=\phi(\vec\epsilon_0)$ at time $T=0$, we have
\[
\phi_0^{*}g_0=a(0,0)\hat{g}+\check{g}(0,0)
\]
By~\eqref{admissible}, this  implies that
\[
\vec\epsilon_0= \big(b_1(0),b_2(0),\dots, b_{n+1}(0)\big)+\mc O\big(\textstyle\sum_{k=1}^{n+1} b_{k}^2(0)\big).
\]
Furthermore, because $\check g(0,0)\in\Sigma_{\alpha_0,\dots,\alpha_\ell,}^C$, we may assume its 
Sobolev norms satisfy
\begin{equation}	\label{GoodStart}
\sum_{|\alpha|=0}^{2n}\|\nabla^{\alpha}\check{g}(0,0)\|_2\leq 2 \delta_0,
\end{equation}
where $\delta_0$ is the constant defined in~\eqref{def:delta01}.

Next we make certain observations about the optimal diffeomorphisms $\phi_{T}=\phi(\vec\epsilon_{T})$
whose existence is proved in Lemma~\ref{LM:endPoint}, at least for $T$ sufficiently small. By
construction of $\phi_0$ and continuity of the solution, we know that there exists a time $M_{1}>0$ such that
for any time $T\in [0,M_1]$, the vector $\vec\epsilon_{T}$ satisfies
\begin{equation}		\label{eq:vecEpsi34}
\|\vec\epsilon_{T}\|\leq \delta_{0}^{\frac{3}{4}},
\end{equation}
where $\delta_0$ is the constant in \eqref{def:delta01}.
This gives some control on the initial condition,
\[
\phi_{T}^{*}g_0=\Big\{a(0,T)+\sum_{k=1}^{n+1} b_{k}(0,T)x^{k}\Big\}\hat{g}+\check{g}(0,T),
\]
where $\check{g}(0,T)\perp \hat{g},\ x^{k}\hat{g}$, and the Sobolev 
norms of $\check g(0,T)$ satisfy
\begin{equation}	\label{2n-derivatives}
\sum_{|\alpha|=0}^{2n}\|\nabla^{\alpha}\check{g}(0,T)\|_{2}\leq 2\delta_0.
\end{equation}
For the evolving solution
\[
g(t,T)=\Big\{a(t,T)+\sum_{k=1}^{n+1}b_k(t,T)x^{k}\Big\}\hat{g}+\check{g}(t,T),
\]
we have the following estimates.

\begin{lemma}	\label{LM:endPo}
In the time interval $t\in [0,T]$, with $T\leq M_1$, if we have
\begin{equation}	\label{eq:bootstrap}
\sum_{|\alpha|\leq 2n} \|\nabla^{\alpha}\check{g}(t,T)\|_{2}\leq \delta_{0}^{\frac{2}{3}}e^{-\frac{2}{3}t},
\end{equation}
then there exists a constant $c_1$, independent of $t$ and $T$, such that
\begin{subequations}
\begin{align}
|b_k(t,T)|&\leq c_1e^{- t} \delta_0^2,\quad\qquad (k=1,\dots,n+1);	\label{five-one}\\
\sum_{|\alpha|\leq 2n}\|\nabla^{\alpha}\check{g}(t,T)\|_{2}&\leq c_1e^{- t}\delta_0;	\label{eq:vecEpsi910}\\
\|\vec\epsilon_{T}\|&\leq \delta_{0}^{\frac{9}{10}}.	\label{five-three}
\end{align}
\end{subequations}
\end{lemma}
Note that \eqref{GoodStart} implies that \eqref{eq:bootstrap} holds at $t=0$.
The lemma is proved in Section~\ref{sec:LMendPo}.

For our bootstrap argument to work, it is vital that estimate~\eqref{eq:vecEpsi910} improves
estimates~\eqref{eq:vecEpsi34} and \eqref{eq:bootstrap}.
Hence by iterating the argument above, we can conclude that Lemma~\ref{LM:endPo} holds for the maximal
time interval in which an optimal admissible diffeomorphism exists, which is $[0,M]$, where $M$ appears in Lemma~\ref{LM:endPoint}.
\smallskip

Recall that at a sequence of times $T\leq M$, we choose optimal diffeomorphisms $\phi_{T}$ and study the
Ricci--DeTurck flow~\eqref{eq:deTurck} with initial condition~\eqref{eq:initial-cond}.
Now we compare various $\phi_{T}^{*}g_0$ for $T\leq M$.
Our conclusion is that if $M=\infty$, then a subsequence of $\{\phi_{T_\gamma}^{*}g_0\}_{\gamma=0}^{\infty}$ is Cauchy.

\begin{lemma}\label{LM:gaugeCon}
For any $T_1,T_2\geq 0$ with $T_2\geq T_1$, one has
\[
\|\check{g}(t,T_2)-\check{g}(t,T_1)\|_{\infty}+\|\phi^{*}_{T_2}g_0-\phi^{*}_{T_1}g_0\|_{\infty}\ls \delta_0^2 e^{-T_1}.
\]
\end{lemma}
The lemma is proved in Section~\ref{sec:LemGauge}.
\smallskip

Next we show that the results above, specifically Lemmas~\ref{LM:endPoint}--\ref{LM:gaugeCon}, hold for $t\in [0,\infty)$.
This follows from our next result, which implies that the maximal value of $M$ is $\infty$.

\begin{lemma}	\label{LM:Minfty}
If there exists a constant $M>0$, such that for any $T\in [0,M]$, there exists an optimal admissible
diffeomorphism $\phi_{T}=\phi(\vec\epsilon_{T}):\ \mathbb{S}^n\rightarrow \mathbb{S}^n$, then there exists a
constant $\delta=\delta(M)>0$ such that for each $T_{1}\in [M,\ M+\delta],$ there exists a unique
$\vec\epsilon_{T_1}\in \mathbb{R}^{n+1}$ such that $\phi(\vec\epsilon_{T_1})$ is an optimal admissible diffeomorphism, and $\check{g}(t,T_1)$ satisfies estimate \eqref{eq:bootstrap}.
\end{lemma}
The lemma is proved in Section \ref{sec:LMMinfty}.
\smallskip

Now we are ready to prove the main result of this section.
\begin{proof}[Proof of Theorem~\ref{THM:conver}]
The fact that the maximal value of $M$ is $\infty$ enables us to use Lemma~\ref{LM:gaugeCon}
to see that the sequences $\{\phi_{T_\gamma}^{*}g_0\}_{\gamma=0}^{\infty}$ and
$\{g(t,T_\gamma)\}_{\gamma=0}^{\infty}$, for any fixed $t$, are Cauchy. 
We denote their limits by $g(0,\infty)$ and $g(t,\infty)$, respectively.
By the results stated above, it is not hard to see that
\begin{align}
\|g(t,\infty)-\hat{g}\|_{\infty}\leq \|g(t,\infty)-g(t,T)\|_{\infty}+\|g(t,T)-\hat{g}\|\lesssim \delta_0 e^{-T}.
\end{align}
By well-posedness of \textsc{nrdf}, $g(t,\infty)$ is the solution of
\[
\Dt g(t,\infty)=-2\Rc[g(t,\infty)]+\mc L_{\hat{W}} [g(t,\infty)]+\frac{2}{n}\bar{R}[g(t,\infty)]\,g(t,\infty).
\]

Finally, we prove that there exists $\vec\epsilon_{\infty}\in \mathbb{R}^{n+1}$ such that
$g(0,\infty)=\phi^{*}(\vec\epsilon_{\infty})g_0$.
We proved that the vectors $\{\vec\epsilon_{T_\gamma}\}_{\gamma=0}^{\infty}$, associated with
$\phi_{T_\gamma}$, are uniformly bounded.
Hence there exists a Cauchy subsequence, again denoted by $\{\vec\epsilon_{T_\gamma}\}_{\gamma=1}^{\infty}$,
such that $\lim_{\gamma\rightarrow \infty}T_\gamma=\infty$ and
$\lim_{\gamma\rightarrow \infty}\vec\epsilon_{T_\gamma}=\vec\epsilon_{\infty}$.
This together with the fact that $\{\phi_{T_\gamma}^{*}g_0\}_{T=0}^{\infty}$ is Cauchy implies the
desired result, which is
\[
\phi_{\infty}^{*}(\vec\epsilon_{\infty})g_0=\lim_{\gamma\rightarrow \infty}\phi_{T_\gamma}^{*}g_0=g(0,\infty).
\]
\end{proof}

\section{Existence and uniqueness of admissible diffeomorphisms}	\label{sec:LMendP}
\begin{proof}[Proof of  Lemma~\ref{LM:endPoint}]
We first show there is a unique admissible diffeomorphism at the initial time. 
That is, we seek a unique $\phi_0$ such that
\[
\phi_0^{*}g_0=a(0,0)\hat{g}+\check{g}(0,0),
\]
with $\check{g}(0,0)\perp \big\{\hat{g},\ x^{k}\hat{g}\big\}$ for $k=1,\dots,n+1$. 

For this purpose, we recall that $g_0$ is of the form 
\begin{align}
g_0=\big\{a(0)+\sum_{k=1}^{n+1}b_{k}(0)x^{k}\big\}\hat{g}+\check{g}(0),
\end{align}
with $|1-a(0)|+\sum_{k}|b_k(0)|\ll 1$ and $\check{g}(0)\perp \hat{g},\ x^{k}\hat{g}$.
As observed in Section~\ref{MainConstruction}, for any vector $\vec\epsilon\in \mathbb{R}^{n+1}$
sufficiently small, there is an admissible diffeomorphism satisfying
\[
\phi^{*}(\vec{\epsilon}) g_0= a(0)\hat{g}
+\sum_{k=1}^{n+1}\big\{a(0)\epsilon_k+b_k\big\}x^{k}\hat{g}
+\phi^{*}\check{g}(0)+\mc O\big(\|\vec\epsilon\|^2+\|\vec\epsilon\| \sum_{k} |b_k|\big).
\]
So we may apply the implicit function theorem to choose a unique $\vec\epsilon\in \mathbb{R}^{n+1}$ that makes
$\phi^{*}(\vec\epsilon)g_0\perp x^{k}\hat{g}$ for $k=1,\dots, n+1$, which yields condition~\eqref{eq:endPo}.

Now we show that there is a unique admissible diffeomorphism $\phi_{T}$ for each sufficiently small time $T.$
To do so, we rewrite~\eqref{eq:deTurck}--\eqref{eq:initial-cond} in the form
\begin{align*}
\Dt g(t,T)&=-2\Rc[g(t,T)]+\mc L_{\hat{W}} [g(t,T)]+\frac{2}{n}\bar{R}[g(t,T)]\,g(t,T),\\
g(0,T)&=\phi^{*}(\vec\epsilon) g_0=\phi_{0}^{*}g_0+\big\{\phi^{*}(\vec\epsilon)g_0-\phi_{0}^{*}g_0\big\}.
\end{align*}
By local well-posedness, we have that, in the interval $t\in [0,T]$,
\[
g(t,T)=g(t,0)+\big[\phi^{*}(\vec\epsilon)g_0-\phi_{0}^{*}g_0\big]+\rho(t)
\]
where $g(t,0)$ is the solution of the Cauchy problem
\begin{align*}
\Dt g(t,0)&=-2\Rc[g(t,0)]+\mc L_{\hat{W}} [g(t,0)]+\frac{2}{n}\bar{R}[g(t,0)]\,g(t,0),\\
g(0,0)&=\phi^*_{0} g_0,
\end{align*}
and where the remainder term $\rho(t)$ satisfies the estimate
\[
\sum_{|\alpha|=0,1}\|\nabla^\alpha \rho(t)\|_{\infty}\ls
Ct \sum_{|\alpha|=0,1,2,3} \|\nabla^{\alpha} [\phi^{*}(\vec\epsilon)g_0-\phi_{0}^{*}g_0]\|_{\infty}.
\]
Hence to make $g(T,T)$ be of the desired form,
\[
g(T,T)=a(T)\hat{g}+\check{g}(T,T),
\]
with $\check{g}\perp \hat{g},\ x^{k}\hat{g}$, we again apply the implicit function theorem.
\end{proof}

\section{Estimates for the diffeomorphisms}	\label{sec:LMendPo}
\begin{proof}[Proof of Lemma~\ref{LM:endPo}]
To derive estimates for the components $b_{k},\ k=1,\dots,n+1$, and $\check{g}$, we begin by linearizing the equation
for $g(t,T)$ around $\hat{g}$. Define $h(t,T)$ by
\[
g(t,T)=\hat{g}+h(t,T).
\]
By the evolution equation~\eqref{eq:deTurck} satisfied by $g(t,T)$, one sees that
\[
\Dt h(t,T)=L[h(t,T)]+N_L[h(t,T)],
\]
where $L$ is the linearized operator defined in~\eqref{Linear} and analyzed in Lemma~\ref{LM:spectral},
and $N_L[h(t,T)]$ is at least second-order in $h$. We note that the quadratic terms in $N_L[h(t,T)]$ are
contractions of the form $h*h$, $h*\nabla h$, $\nabla h*\nabla h$, and $h*\nabla\nabla h$. (For explicit
formulas, see Lemma~29 of \cite{AK07}.)

We decompose $h(t,T)$ according to the spectrum of $L$, writing
\begin{equation}	\label{def:htT}
h(t,T):=\left\{[a(t,T)-1]+\sum_{k=1}^{n+1}b_k(t,T)x^{k}\right\}\hat{g}+\check{g}(t,T).
\end{equation}
Below, we control the various components of $h(t,T)$ in~\eqref{def:htT}.

It is easy to handle the component $a(t,T)-1$ of $h(t,T)$ in~\eqref{def:htT}.
Indeed, in Lemma~\ref{LM:redA}, we observed that it can be controlled by estimate~\eqref{eq:meas},
namely
\[
|a(t,T)-1|=\mc O\big(|b_{k}(t,T)|^2+\|\check{g}\|_{\infty}^2\big).
\]
So the desired estimate for $|a(t,T)-1|$ follows from those derived below for the remaining components.

Now we turn to the terms $b_{k}(t,T)$, which are given by\footnote{Recall that the $L^2$ inner product
$\langle \cdot,\cdot\rangle$ is computed with respect to the standard metric $\sigma$ induced on the unit sphere
$\mb S^n\subset\mb R^{n+1}$.}
\[
b_k(t,T)=\frac{\langle x^{k}\hat{g},\  h(t,T)\rangle}{\langle x^{k} \hat g,\ x^{k}\hat{g}\rangle}.
\]
At $s\in[t,T]$, we apply the inner product $\langle x^{k}\hat{g},\ \cdot \rangle$ to the evolution equation~\eqref{eq:deTurck}, obtaining
\[
\dot{b}_k(s,T)\langle x^{k}\hat{g},\ x^{k}\hat{g}\rangle=(n-2) b_{k}(s,T)\langle x^{k}\hat{g},\ x^{k}\hat{g}\rangle
+\langle x^{k}\hat{g},\ N_L[h(s,T)]\rangle
\]
where the term $\langle x^{k}\hat{g}, N_L[h(s,T)]\rangle$ satisfies the estimate
\begin{equation}	\label{LoseADerivative}
|\langle x^{k}\hat{g},N_L[h(s,T)]\rangle | \ls \sum_{k} |b_k(s,T)|^2 +\sum_{|\alpha|=0,1}\|\nabla^{\alpha} \check{g}(s,T)\|_{2}^2.
\end{equation}
Observe that the right-hand side of~\eqref{LoseADerivative} does not depend on $\nabla^{\alpha}\check{g}$ for $|\alpha|=2$,
even though $N_L[h]$ does. This is because when computing the inner product, we integrate by parts to move one derivative off $h$.
Using~\eqref{LoseADerivative}  together with the fact that $b_{k}(T,T)=0$, we see that for any time $t\in [0,T]$, we have
\[
b_{k}(t,T)=-\int_{t}^{T}e^{-(n-2)(s-t)}\frac{\langle x^{k}\hat{g},N_L[h(s,T)]\rangle}{\langle x^{k} \hat g,\ x^{k}\hat{g}\rangle}\,\mr ds.
\]
Consequently, we have
\begin{align}
|b_{k}(t,T)|\ls
\int_{t}^{T} e^{-(n-2)(s-t)}\left\{\sum_{k=1}^{n+1} |b_k(s,T)|^2
+\sum_{|\alpha|=0,1} \|\nabla^{\alpha}\check{g}(s,T)\|_2^2\,\|\check{g}(s,T)\|_{\infty}\right\}\mr ds.\label{eq:bktT}
\end{align} 
Next we use assumption~\eqref{eq:bootstrap} in Lemma \ref{LM:endPo}, which is
$\sum_{|\alpha|\leq 2n} \|\nabla^{\alpha}\check{g}\|_{2}\leq \delta_{0}^{\frac{2}{3}}e^{-\frac{2}{3}t}$,
and Sobolev embedding, $\|\check{g}\|_{\infty}\ls \sum_{|\alpha|\leq n} \|\nabla^{\alpha}\check{g}\|_{2}$,
to obtain
\[
|b_{k}(t,T)|\ls
\int_{t}^{T} e^{-(n-2)(s-t)}\left\{\sum_{k=1}^{n+1} |b_k(s,T)|^2
+\delta_0^2 e^{-2s}\right\}\mr ds.
\]
Recall that $n-2\geq 1.$
By a standard fixed point argument we find there is $c>0$ such that
\begin{align}
|b_{k}(t,T)|\leq c e^{-t} \delta_0^2.\label{eq:bktT2},
\end{align}
which gives estimate~\eqref{five-one}.

To prove estimate~\eqref{eq:vecEpsi910}, we analyze the component $\check{g}$ of $h(t,T)$ in \eqref{def:htT}.
By Lemma~\ref{LM:spectral} and \eqref{eq:bktT2} above, one has
\begin{align*}
\frac{d}{dt} \langle \check{g},\check{g}\rangle
&=2\langle \check{g},L\check{g}\rangle+2\langle \check{g},N_L[h]\rangle\\
&\leq -8 \langle \check{g},\check{g}\rangle
+\mc O\left(\sum_{k=1}^{n+1}|b_{k}|^2+\sum_{|\alpha|=0,1}\|\nabla^{\alpha} \check{g}\|_2^3\right)\\
&\leq -8 \langle \check{g},\check{g}\rangle
+\mc O\left(
\delta_0^4 e^{-2t }+\sum_{|\alpha|=0,1}\|\nabla^{\alpha} \check{g}\|_2^3\right)
\end{align*}
Therefore, 
\[
\frac{d}{dt}\langle \check{g},\check{g}\rangle\leq e^{-8t} \langle \check{g}(0),\check{g}(0)\rangle+\int_{0}^{t} e^{-8(t-s)}
\mc O\left(\delta_0^4 e^{-2s }+\sum_{|\alpha|=0,1}\|\nabla^{\alpha} \check{g}\|_2^3\right)\mr ds.
\]

In similar fashion, one computes that, for any $j\leq n,$
\begin{align}
\frac{d}{dt} \langle L^{j} \check{g},\ L^{j} \check{g}\rangle=2\langle L^{j} \check{g},\ L^{j+1} \check{g}\rangle+2\langle L^{j} \check{g},\ L^{j}N_{L}(h)\rangle.\label{target}
\end{align}
For the first term on the right-hand side, we claim that for a $\kappa>0$ to be chosen below, one has
\begin{equation}	\label{firstT}
2\langle L^{j} \check{g},\ L^{j+1} \check{g}\rangle\leq
-\kappa \sum_{|l|\leq 2j+1}\|\nabla^{l}\check{g}\|_{2}^2-3\langle L^{j}\check{g},\ L^{j}\check{g}\rangle.
\end{equation}
To see this, we decompose $2\langle L^{j} \check{g},\ L^{j+1} \check{g}\rangle$ into two terms,
\[
2\langle L^{j} \check{g},\ L^{j+1} \check{g}\rangle=
2\kappa \langle L^{j} \check{g},\ L^{j+1} \check{g}\rangle+2(1-\kappa)\langle L^{j} \check{g},\ L^{j+1} \check{g}\rangle.
\]
For the first term, we use the derivatives in $L^{j}$ to get $C_1,\ C_2>0$ such that
\[
2\kappa \langle L^{j} \check{g},\ L^{j+1} \check{g}\rangle\leq
-C_1\kappa \sum_{|\alpha|\leq 2j+1}\|\nabla^{\alpha}\check{g}\|_{2}^2
+C_2\kappa \|\check{g}\|_2^2.
\]
This, together with the estimate for the second term
\[
\langle L^{j} \check{g},\ L^{j+1} \check{g}\rangle\leq -4\langle L^{j} \check{g},\ L^{j}\check{g}\rangle
\]
implies the claimed estimate \eqref{firstT}.

For the last term on the right-hand side of~\eqref{target}, 
we have to be careful about the number of derivatives on $\check{g}$.
Integrating by parts and using estimate~\eqref{eq:bktT2}, we obtain
\begin{align}
|\langle L^{j} \check{g},\ L^{j}N_{L}(h)\rangle|&\ls
\sum_{k=1}^{n+1}|b_k|^2+\sum_{|\alpha|\leq 2j+1}\|\nabla^{\alpha}\check{g}\|_{2}^2
\left\{\sum_{|\beta|\leq j} \|\nabla^{\beta}\check{g}\|_{\infty}+\sum_{k=1}^{n+1}|b_k|\right\}\nonumber\\
&\ls\delta_0^{4}e^{-2t}+\sum_{|\alpha|\leq 2j+1}\|\nabla^{\alpha}\check{g}\|_{2}^2
\left\{\sum_{|\beta|\leq j} \|\nabla^{\beta}\check{g}\|_{\infty}+\delta_0^2 e^{-t}\right\}.
\end{align}
To control $\sum_{|\beta|\leq j} \|\nabla^{\beta}\check{g}\|_{\infty}$ for $j\leq n$,
we use Sobolev embedding along with assumption~\eqref{eq:bootstrap} in Lemma \ref{LM:endPo} to obtain
\begin{align}
\sum_{|\beta|\leq n} \|\nabla^{\beta}\check{g}\|_{\infty}\ls\sum_{|\beta|\leq 2n} \|\nabla^{\beta}\check{g}\|_{2}
\leq \delta_{0}^{\frac{2}{3}}\ll 1.
\end{align}
Combining this with \eqref{firstT}, and choosing $\kappa\gg\delta_0^{\frac{2}{3}}$, one gets
\[
\frac{d}{dt}\sum_{j=1}^{n}\langle L^{j}\check{g},\ L^{j}\check{g}\rangle \leq 
-6\sum_{j=1}^{n}\langle L^{j}\check{g},\ L^{j}\check{g}\rangle
+\mc O\left(\delta_0^{4}e^{-2t}\right).
\]
This estimate immediately implies that
\begin{multline}	\label{eq:bootG}
\sum_{j=1}^{n}\langle L^{j}\check{g}(t,T),\ L^{j}\check{g}(t,T)\rangle\leq
e^{-6t} \sum_{j=1}^{n}\langle L^{j}\check{g}(0,T),\ L^{j}\check{g}(0,T)\rangle\\
+\int_{0}^{t} e^{-6(t-s)} \mc O\left(\delta_0^{4}e^{-2s}\right)\,\mr ds.
\end{multline}
By estimate~\eqref{2n-derivatives} for $\check{g}(0,T)$, we find there exists $c>0$ such that
\[
\sum_{|\alpha|\leq 2n}\|\nabla^{\alpha}\check{g}(t,T)\|_2\ls
\sum_{j=1}^{n}\langle L^{j}\check{g}(t,T),\ L^{j}\check{g}(t,T)\rangle
\leq c\delta_0^2 e^{-2t},
\]
which yields estimate~\eqref{eq:vecEpsi910}.

Finally, we prove the last estimate in Lemma \ref{LM:endPo}.
By the construction of $b_{k}(0,T)$ and our estimates above for $b_{k}(0,T)$ and $b_k(0,0)$, we observe that
\[
\vec\epsilon_{T}=\mc O \left(\sum_{k}|b_{k}(0,T)+b_{k}(0,0)|\right)=\mc O (\delta_0),
\]
which yields~\eqref{five-three}.
\end{proof}

\section{Existence of a convergent subsequence}	\label{sec:LemGauge}
\begin{proof}[Proof of Lemma~\ref{LM:gaugeCon}]
The basic idea is to compare solutions $g(t,T_2)$ and $g(t,T_1)$ of the DeTurck-Ricci flow~\eqref{eq:deTurck}
in the time interval $[0,\ T_1].$

Recall that at time $t=T_1$, $g_{T_2}$ and $g_{T_1}$ are of the form
\begin{align}
g(T_1,T_2)=&(a(T_1,T_2)+\sum_{k=1}^{n+1}b(T_1,T_2)x^{k})\hat{g}+\check{g}(T_1,T_2),\\
g(T_1,T_1)=&a(T_1,T_1)\hat{g}+\check{g}(T_1,T_1),
\end{align}
where $\check{g}(T_1,T_1)$ and $\check{g}(T_1,T_2)$ satisfy appropriate orthogonality conditions.

We begin by deriving an evolution equation for 
\[
\Psi(t):=g(t,T_2)-g(t,T_1)
\]
from those for $g(t,T_2)$ and $g(t,T_1)$ given in \eqref{eq:deTurck}, obtaining
\begin{subequations}
\begin{align}
\Dt \Psi&=L\Psi+R_L(\Psi),	\label{eq:eqPsi}\\
\Psi(0)&=g(0,T_2)-g(0,T_1)=\phi^{*}_{T_2}g_0-\phi^{*}_{T_1}g_0,
\end{align}
\end{subequations}
where $L$ is the linear operator in~\eqref{GeneralLinear},
\[
L\Psi=\Delta_{\ell}\Psi+2(n-1) \Big\{\Psi-\frac{1}{n} \bar{H}_{\Psi}\hat{g}\Big\},
\]
and $R_L$ denotes the nonlinear ``remainder term.''
In the time interval $[0,\ T_1]$, the remainder term satisfies
\begin{multline}	\label{eq:coRemainder}
\|R_L(\Psi)\|_{\infty}\ls \sum_{|\alpha|=0,1} \|\nabla^{\alpha} \Psi\|_{\infty}\\
*\left\{\sum_{|\beta|=0,1,2} \|\nabla^{\beta} (g(t,T_1)-\hat{g})\|_{\infty}
+\sum_{|\beta|=0,1,2} \|\nabla^{\beta} (g(t,T_2)-\hat{g})\|_{\infty}\right\}.
\end{multline}

Wishing to exploit the spectrum of $L$ using Lemma~\ref{LM:spectral},
we decompose $\Psi$ as
\begin{equation}	\label{eq:decomPsi}
\Psi(t)=\Big\{\delta(t)+\sum_{k=1}^{n+1}\beta_k(t) x^{k}\Big\}\hat{g}+\eta(t),
\end{equation}
with 
\[
\eta\perp \hat{g},\,x^{k}\hat{g}\qquad\mbox{for}\quad k=1,\dots,n+1.
\]
To control the various components of $\Psi$ in~\eqref{eq:decomPsi}, the
key steps are to prove that the components corresponding to $\beta_{k}(t)$ dominate, and
that their evolutions are almost linear. In what follows, we outline the proof, omitting some
details that are easily established by standard methods.

To begin, we recall that $\phi_{T_{m}}=\phi(\vec\epsilon_{T_m})$ for some
$\vec\epsilon_{T_m}\in \mathbb{R}^{n+1}\ (m=1,2)$. Then we observe that 
\begin{align*}
\Psi(0)&=\phi_{T_2}^{*}g_0-\phi_{T_1}^{*}g_0\\
&= a_0\big\{\phi_{T_2}^{*}\hat{g}-\phi_{T_1}^{*}\hat{g}\big\}
	+\sum_{k=1}^{n+1}b_{k}(0)\big\{\phi_{T_2}^{*}x^{k}\hat{g}-\phi_{T_1}^{*}x^{k}\hat{g}\big\}
	+\phi_{T_2}^{*}\check{g}-\phi_{T_1}^{*}\check{g}\\
&=\Big\{\omega +\sum_{k=1}^{n+1} \big[\epsilon_k(T_1)-\epsilon_{k}(T_2)+\mu_k\big]\Big\}x^{k}\hat{g}+\Phi.
\end{align*}
Here $\omega$, $\mu_k\in \mb R$, and $\Phi$ satisfy the estimates
\[
|\omega|+\sum_{k=1}^{n+1}|\mu_k|+\|\Phi\|_{2}\ls \|\vec\epsilon(T_1)-\vec\epsilon(T_2)\|_{2}\,\delta_0,
\]
where $\delta_0\in \mathbb{R}^{+}$ is defined in \eqref{def:delta01} as
\[
\delta_0:=|1-a(0)|+\sum_{k=1}^{n+1}|b_k(0)|+\sum_{|\alpha|=0}^{2n} \|\nabla^{\alpha} \check{g}(0)\|_2.
\]
We recall too that $g_0$ decomposes as
\[
g_0=\big\{a(0)+\sum_{k=1}^{n+1} b_{k}(0)x^{k}\big\}\hat{g}+\check{g}(0).
\]
These observations show that the initial value $\Psi(0)$ is dominated
by its components in the directions $x^{k}\hat{g}$. That is,
\[
\Psi(0)\approx \sum_{k=1}^{n+1} \big\{\epsilon_k(T_1)-\epsilon_{k}(T_2)\big\} x^{k}\hat{g}.
\]
Comparing this to the decomposition of $\Psi(t)$ in \eqref{eq:decomPsi}, one sees that
\[
(\beta_1(0),\dots,\beta_{n+1}(0))\approx \vec\epsilon_k(T_1)-\vec\epsilon_{k}(T_2).
\]

For any later time $t\in [0,T_1],$ we claim that the terms $\sum_{k=1}^{n+1}\beta_{k}x^{k}\hat{g}$
continue to dominate $\Psi(t)$. The claim follows easily from two facts. \textsc{(i)} The smallness
of $R_L(\Psi)$ in estimte~\eqref{eq:coRemainder} means that equation~\eqref{eq:eqPsi} is dominated
by its linear term. And \textsc{(ii)} the spectral decomposition of  $L$ in Lemma \ref{LM:spectral} shows
that the fastest growing eigenspace is spanned by $\beta_k, \ k=1,2,\dots, n+1$, with eigenvalue $n-2\geq1$.
All other eigenvalues are non-positive. By standard arguments, these facts imply that in the time interval $t\in [0,\ T_1],$
one has
\[
(\beta_1(t),\ \beta_{2}(t),\ \dots,\ \beta_{n+1}(t))\approx e^{(n-2)t} (\beta_1(0),\ \beta_{2}(0),\ \dots,\ \beta_{n+1}(0)).
\]
Combining this with the estimates for $b_k(t, T_1)$ and $b_{k}(t,T_2)$ proved in Lemma~\ref{LM:endPo} gives
the estimate in the statement of Lemma~\ref{LM:gaugeCon}.
\end{proof}

\section{Long-time existence}	\label{sec:LMMinfty}
\begin{proof}[Proof of Lemma~\ref{LM:Minfty}]
To extend the existence of optimal admissible diffeomorphisms after $M<\infty$, we use our estimates
for the solution $g(t,M)$ of \textsc{nrdf} with initial condition $\phi^{*}(\vec{\epsilon}_{M})g_0$ to find
an $\vec{\epsilon}_{M+\delta}\in\mb R^{n+1}$ determining an optimal admissible diffeomorphism at
time $M+\delta$, for some $\delta>0$ sufficiently small.

As in equation~\eqref{eq:endPo}, we decompose the solution $g(t,M)$ as
\[
g(t,M)=\Big\{a(t,M)+\sum_{k} b_k(t,M) x^k\Big\}\hat{g}+\check{g}(t,M),
\]
with 
\[
\check{g}(t,M)\perp \hat{g},x^{k}\hat{g}.
\]
The estimates in Lemma~\ref{LM:endPo} for $g(M,M)$ show that $g(t,M)$ can be continued
past $M$, at least for a short time. So by the consequence of Lemma~\ref{LM:endPoint}
that all $b_k(M,M)=0$, and local well-posedness of \textsc{nrdf}, we can
choose $\delta>0$ sufficiently small such that in the time interval $[M, \ M+\delta]$, the quantities
$\sum_{k}|b_k(t,M)|$ $(k=1,\dots,n+1)$ and $\|\nabla^{\alpha}\big[\hat g(t,M)-\hat g(M,M)\big]\|_2$
$(|\alpha|\leq2n)$ are as small as one likes. Because of this, the construction proceeds exactly
as in the proof of Lemma~\ref{LM:endPoint}. Hence we omit the details.
\end{proof}


\begin{thebibliography}{9}

\bibitem {AK07}
    \textbf{Angenent, Sigurd B.; Knopf, Dan.}
    Precise asymptotics of the Ricci flow neckpinch.
    \emph{Comm.~Anal.~Geom.}~\textbf{15} (2007), no.~4, 773--844.

\bibitem{BS}
	\textbf{Brendle, Simon; Schoen, Richard.}
	Manifolds with $1/4$-pinched curvature are space forms.
	\emph{J.~Amer.~Math.~Soc.}~{\bf 22} (2009), no.~1, 287--307.

\bibitem{E68}
  \textbf{Ebin, David G.}
  The manifold of Riemannian metrics.
  \emph{Proc.~Sympos.~Pure Math.~(Global Analysis),} \textbf{15} (1968) 11--40.

\bibitem{FH78}
  \textbf{Farrell, Thomas; Hsiang, W.C.}
  On the rational homotopy groups of the diffeomorphism groups of discs, spheres and aspherical manifolds.
  \emph{Proc.~Sympos.~Pure Math.~XXXII}~\textbf{32}, Amer.~Math.~Soc., Providence, R.I., 1978, 325--337.

\bibitem{FO} 
\textbf{Farrell, Thomas; Ontaneda, Pedro.}
The Teichm\"{u}ller space of pinched negatively curved metrics on a hyperbolic manifold is not contractible.
\emph{Ann.~of Math.}~(2) {\bf 170} (2009), 45--65.

\bibitem{GKS}
\textbf{Gang, Zhou; Knopf, Dan; Sigal, Israel Michael.}
Neckpinch dynamics of asymmetric surfaces evolving by mean curvature flow.
(\texttt{arXiv:1109.0939})

\bibitem{GK}
\textbf{Gang, Zhou; Knopf, Dan.}
Universality in mean curvature flow neckpinches.
\emph{ Duke Math.~J.}~To appear.
(\texttt{arXiv:1308.5600})

\bibitem{Hatcher}
  \textbf{Hatcher, Allen E.}
  A proof of the Smale conjecture, $\Df(\bS^3)\simeq\mr O(4)$.
  \emph{Ann.~of Math. (2)} \textbf{117} (1983), no.~3, 553--607. 

\bibitem{K09}
\textbf{Knopf, Dan.}
Convergence and stability of locally $\mathbb{R}^{N}$-invariant solutions of Ricci flow.
\emph{J.~Geom.~Anal.}~\textbf{19} (2009), no.~4, 817--846.

\bibitem{Smale}
  \textbf{Smale, Stephen.}
  Diffeomorphisms of the $2$-sphere. 
  \emph{Proc.~Amer.~Math.~Soc.} \textbf{10} (1959) 621--626. 

\end{thebibliography}
\end{document}